\documentclass[12pt]{amsart}

\usepackage[centertags]{amsmath}
\usepackage{amsfonts}
\usepackage{amscd}
\usepackage[sc,osf,slantedGreek]{mathpazo}
\usepackage{euscript}
\usepackage{graphics}
\usepackage{color}
\usepackage{parskip}
\usepackage[all]{xy}
\usepackage[T1]{fontenc}
\usepackage{parskip}

\newcommand{\B}[1]{{\mathbb #1}}
\newcommand{\C}[1]{{\EuScript #1}}

\newtheorem{theorem}[subsection]{Theorem}%[section]
\newtheorem{corollary}[subsection]{Corollary}
\newtheorem{lemma}[subsection]{Lemma}
\newtheorem{proposition}[subsection]{Proposition}
\theoremstyle{definition}

\newtheorem{example}[subsection]{Example}
\theoremstyle{remark}
\newtheorem{remark}[subsection]{Remark}

\numberwithin{figure}{section}
\numberwithin{table}{section}
\newcommand{\Om}{{\Omega}}

\newcommand{\Mo}{(M,\omega )}

\newcommand{\cp}{\B C\B P}
\newcommand\OP{\operatorname}

\newcommand\Diff{\operatorname{Diff}}
\newcommand\Homeo{\operatorname{Homeo}}

\newcommand\Ham{\operatorname{Ham}}

\begin{document}

\title[Classifying spaces of groups of homeomorphisms] {On the
cohomology of classifying spaces of groups of homeomorphisms}
\author{Jarek K\k edra} 
\address{University of Aberdeen and University of Szczecin}
%\curraddr{      }
\email{kedra@abdn.ac.uk}

%
%
%\begin{abstract}

%\end{abstract}

\maketitle

%\tableofcontents

%%%%%%%%%%%%%%%%%%%%%%%%%%%%%%%%%%%%%%%%%%%%%%%%%%%%%%
\section{Introduction and statement of the results}\label{S:intro}
%%%%%%%%%%%%%%%%%%%%%%%%%%%%%%%%%%%%%%%%%%%%%%%%%%%%%%

Let $M$ be a closed simply connected $2n$--dimensional manifold. 
The present paper is concerned with the cohomology of classifying
spaces of connected groups of homeomorphisms of $M$.

\subsection{Conventions} We make the following standing assumptions
throughout the paper.  The topology on a group of homeomorphisms of a
manifold is assumed to be compact-open. We consider the cohomology
with real coefficients unless otherwise specified.

\subsection{Generic coadjoint orbits}\label{SS:orbits}

\begin{theorem}\label{T:orbits}
Let $M=G\cdot \xi\subset \mathfrak g^{\vee}$ be a {\bf generic}
coadjoint orbit of a compact connected semisimple Lie group $G.$
Suppose that the action $G\to \OP{Homeo}(M)$ has a finite kernel.
Then the homomorphism
$$
H^*(B\OP{Homeo}_0(M)) \to H^*(BG)
%H^*(B\C H) \to H^*(BG)
$$
induced by the action is surjective. 
It is surjective in degree four for every (not necessarily
generic) coadjoint orbit.
\end{theorem}

The proof of this theorem is given in Section \ref{SS:proof_orbits}.
{\bf{\em Genericity}} means that there exists a non-empty Zariski open
subset $Z \subset \mathfrak~g^{\vee}$ of the dual of the Lie algebra
of $G$ such that the theorem holds for an orbit $G\cdot \xi$, where
$\xi \in Z$. We shall discuss the generic orbits in Section
~\ref{S:examples}. The examples will include complex projective spaces
and flag manifolds. However, the subset $Z$ for a given group $G$ is
not understood.

\begin{remark}\label{R:monoid}
Theorem \ref{T:orbits} and most of the results in this paper can be
directly generalised to a connected topological monoid of homotopy
equivalences, cf. K\k edra--McDuff ~\cite{MR2115670}.
\end{remark}

The following result is an immediate consequence of Theorem \ref{T:orbits}.
\begin{corollary}\label{C:orbits}
Let $M=G\cdot \xi$ be as in Theorem \ref{T:orbits}.
Let $\C H \subset \OP{Homeo}(M)$ be a connected group of homeomorphisms
containing $G$ as a subgroup. Then the induced homomorphism
$$
H^*(B\C H) \to H^*(BG)
$$
is surjective for a generic coadjoint orbit $M$ and in degree four it
is surjective for all orbits. \qed
\end{corollary}

The most important examples of groups $\C H$ to which we apply the
above result are the group $\Ham\Mo$ of Hamiltonian diffeomorphisms,
$\Diff_0(M)$ -- the connected component of the identity of
the group of diffeomorphism, and $\Homeo_0(M)$ -- the connected component
of the identity of the group of homeomorphisms.

\begin{example}
Let $M=\cp^{n_1}\times \ldots \times \cp^{n_k}$ be equipped with a
product symplectic form $\omega$ invariant under the natural action of
the product of special unitary groups. Then the induced homomorphism
$$
H^*(B\Ham\Mo)\to H^*(B(\OP{SU}(n_1+1)\times \ldots \times \OP{SU}(n_k+1))
$$ 
is surjective. In particular, $\dim H^4(B\Ham\Mo)\geq k.$ 

We know from \cite{MR2115670} that the relevant characteristic classes
don't vanish on spheres and hence we also have that
$$
\OP{rank}(\pi_{2m+1}(\Ham\Mo))\geq k\, \text{ for  }\,
m\leq \min \{n_1,\ldots ,n_k\}.
$$
\qed
\end{example}

Combining the above example (for $k=2$) 
with a theorem of Seidel~\cite{MR2000m:53124}
we immediately obtain the following result.

\begin{theorem}\label{T:seidel}
Let $\cp^m\times \cp^n$ be equipped with a product symplectic
form $\omega$ such that the symplectic areas of the lines in factors are
equal. Suppose that $m\geq n\geq k\geq 1$. Then
$$
\OP{rank}(\pi_{2k+1}(\Diff(\cp^m\times \cp^n)))\geq 2n-k+2.
$$
\qed
\end{theorem}

\subsection{Circle actions}\label{SS:circle}
A circle action $\B S^1 \to \C H\subseteq \OP{Homeo}(M)$ is called 
$\C H$-{\bf{\em inessential}} if
it defines a contractible loop in $\C H.$ For example, if a
connected simply connected Lie group $G$ acts on $M$ then a circle
subgroup of $G$ yields an inessential circle action.
The following result generalises the second part
of Theorem \ref{T:orbits}.

\begin{theorem}\label{T:circle}
Let $(M_i,\omega_i)$, $i=1,2,\ldots,m$ be a closed simply connected
symplectic manifold admitting an inessential nontrivial Hamiltonian
circle action. Let $M=M_1\times M_2\times \dots \times M_m$ be
equipped with a product symplectic form $\omega$.  
If $\C H$ is a connected group containing the product
$\Ham(M_1,\omega_1)\times \ldots \times \Ham(M_m,\omega_m)$ then 
$ \dim H^4(B\C H;\B R) \geq~m$ and 
$\OP{rank} \pi_3(\C H;\B R) \geq m.$
\end{theorem}

\subsection{The fibre integral subalgebra}
Let $G\to \OP{Homeo}_0(M)$ be an action of a connected
topological group.  In Section \ref{S:strategy}, we define a certain
graded subalgebra $\B A^*_G(M) \subset H^*(BG)$ associated with the
action. It is called the fibre integral subalgebra and it can be
calculated in certain cases. It is our main technical tool and its
basic properties are presented in Section \ref{S:basic}.

\subsection{Relation to the previous work}\label{SS:history}
The obvious strategy to understand the topology of the classifying
space $B \C H$ of a homeomorphism group $\C H$ is to consider a map 
$f\colon B\to B\C H$ defined on a space with understood topology and,
for example, examine the induced map on the cohomology. In the present
paper we mostly investigate the homomorphism $H^*(B\C H) \to H^*(BG)$
for the natural action of a compact Lie group $G$ on a homogeneous
space $G/H.$ Conjecturally, the homomorphism
$$
H^*(B\Homeo_0(G/H)) \to H^*(BG)
$$
should be surjective for the real cohomology provided the
action of $G$ is effective.

Apart from the classical results about diffeomorphisms of low
dimensional spheres and surfaces the first such surjectivity result
was obtained by Reznikov in \cite{MR2000f:53116}. He proved that the
natural Hamiltonian action of $\OP{SU}(n)$ on the complex projective
plane $\cp^{n-1}$ induces the surjection $H^*(B\Ham(\cp^{n-1}))\to
H^*(B\OP{SU}(n)).$ He proved it by using a Hamiltonian version of the
Chern--Weil theory. He also conjectured that a similar statement
should be true for other coadjoint orbits.

The result of Reznikov was improved and generalised to flag manifolds
by K\k edra--McDuff in \cite{MR2115670}. We proved that the
characteristic classes defined by Reznikov are in fact topological in
the sense that they can be defined in the cohomology ring of the
topological monoid of homotopy equivalences of a symplectic manifold. 

The algebraic independence of the Reznikov classes was proved by
Gal--K\k edra--Tralle in \cite{gkt} for a generic coadjoint orbit of a
semisimple Lie group. It was shown by examples that his classes cannot
be algebraically independent in general (see Example \ref{E:g24}).

The results in \cite{gkt,MR2115670} and also \cite{jk} are
applications of the fibre integral method. That is, certain
characteristic classes are defined as fibre integrals. This is the
main tool here as well. The new ingredient is
that we consider distinct symplectic forms on a given manifold at the
same time.  More precisely, the Reznikov characteristic classes are
equal to the fibre integrals of powers of a certain universal
cohomology class called the coupling class. This class is induced by a
fixed symplectic form.  In this paper we consider fibre integrals of
products of many coupling classes induced by distinct symplectic
forms.

\subsection*{Acknowledgements}
The present work is built upon papers \cite{gkt,MR2115670}.  I thank my
coauthors Dusa McDuff, \'Swiatos\l aw Gal and Alex Tralle for
discussions. I thank Dusa McDuff and Old\v{r}ich Sp\'{a}\v{c}il for
useful comments on a preliminary version of this paper.
Any remaining mistakes are of my responsibility. 

%%%%%%%%%%%%%%%%%%%%%%%%%%%%%%%%%%%%%%%%%%%%%%%%%%%%%%%%%%%%%%%%%
\section{A strategy and few technical results}\label{S:strategy}
%%%%%%%%%%%%%%%%%%%%%%%%%%%%%%%%%%%%%%%%%%%%%%%%%%%%%%%%%%%%%%%%

\subsection{Fibre integration}\label{SS:fibre_integration}
Let $M\to E \stackrel{\pi}\to B$ be an oriented bundle with closed
$n$-dimensional fibre. There is a homomorphism of $H^*(B)$-modules
$$
\pi_!\colon H^{n+k}(M)\to H^k(B).
$$
It is defined to be the composition
$$
H^{n+k}(E) \to E_{\infty}^{k,n}\to E_2^{k,n} = H^k(B;H^n(M)) = H^k(B),
$$
where the $E^{p,q}_m$ is an $m$-th term of the associated Leray-Serre spectral
sequence. The property that the fibre integration is a morphism of
$H^*(B)$-modules means that
$$
\pi_!(\alpha \cdot \pi^*(\beta)) = \pi_!(\alpha)\cdot \beta.
$$

Moreover, fibre integration is multiplicative with respect to
the cross product. More precisely, let
$p_1\colon E_1\to B_1$ and $p_2\colon E_2\to B_2$ be
oriented bundles with closed fibres. Then

$$
(p_1 \times p_2)_!(\alpha \times \beta) = 
(p_1)_!(\alpha) \times (p_2)_!(\beta).
$$
This multiplicativity easily follows from the definition and good
properties of the spectral sequence.
\subsection{A very general view}

Let $G$ be a topological group acting on a closed oriented
$n$-manifold $M.$ Consider the associated universal fibration
$$
M\stackrel{i}\to M_G \stackrel{\pi}\to BG
$$ 
induced by the action.  
%
%Suppose that there exists a cohomology class
%$\Omega \in H^n(M_G)$ such that $\langle i^*(\Omega),[M]\rangle =~1.$
%It then follows that the integration along the fibre is
%surjective. Indeed, we have $\beta = \pi_!(\pi^*(\beta)\cdot \Omega),$
%due to the fact that the fibre integration is a morphism of
%$H^*(B)$-modules.
%
Given a subalgebra $\B A\subset H^*(M_G)$ we consider a subalgebra
$$
\langle \pi_!(\B A)\rangle \subset  H^*(BG)
$$
generated by the fibre integrals of elements from $\B A.$ The strategy
is to choose an appropriate subalgebra $\B A$ for which one can make
computations.

\subsection{The fibre integral subalgebra}

Let us assume that $M$ is simply connected and $G$ is connected and
let $\B A^* = \langle H^2(M_G) \rangle \subset H^*(M_G)$ be the
subalgebra generated by the classes of degree 2.  Define the 
{\bf {\em fibre integral subalgebra}}
$$
\B A^*_G(M) := \langle \pi_!(\B A^*) \rangle
$$
associated with the action of $G$ on $M$ to be the graded subalgebra
of $H^*(BG)$ generated by the fibre integrals of the products of
cohomology classes of degree $2.$ In particular, an element of 
$\B A^{2k}_G(M)$ is a linear combination of classes of the form
$\pi_!(a_1\ldots a_{n+k})$ where $a_i\in H^2(M_G)$ and $\dim M=2n.$ We
say that the fibre integral subalgebra is {\bf{\em full}} if it is equal to
the whole of $H^*(BG).$

\subsection{The surjectivity lemma}
Let $\C H \subset \OP{Homeo}(M)$ be a connected group of
homeomorphisms of a simply connected manifold $M.$ Let
$$
M\stackrel{i}\to M_{\C H}\stackrel{\pi}\to B\C H
$$ be the
universal bundle associated with the action of $\C H$ on $M.$

\begin{lemma}\label{L:surjectivity}
The homomorphism $i^*\colon H^2(M_{\C H})\to H^2(M)$ induced by the
inclusion of the fibre is surjective.
\end{lemma}

\begin{proof}
Consider the associated Leray-Serre spectral sequence $E_m^{p,q}.$
Since $M$ and $B\C H$ are simply connected both the first row and
the first column of the sequence are trivial. That is,
$$
E_2^{p,q}=H^p(B\C H)\otimes H^q(M)=0
$$ 
if $p=1$ or $q=1.$

Let $a\in H^2(M)=E_2^{0,2}$ be a nonzero cohomology class.  Since
$E_2^{2,1}=~0$ we have $d_2(a)=0$. Thus to finish the proof we need to
show that $d_3(a)=0.$

Since $M$ is finite dimensional there exists a number $k\in \B N$ such
that $a^k\neq 0$ and $a^{k+1}=0.$ Observe that
$$
E_3^{3,2k}\subset E_2^{3,2k}=H^3(B\C H)\otimes H^{2k}(M)
$$ 
because the differential $d_2\colon E_2^{1,2k+1}\to E_2^{3,2k}$ is
trivially zero, for $E_2^{1,2k+1}=0.$ Thus the following computation
$$
0=d_3(a^{k+1}) = (k+1)\,d_3(a)\otimes a^k 
$$
implies that $d_3(a) = 0$ are required.
\end{proof}

\begin{remark}
The above argument proves in fact the following. Suppose that $F\to
E\to B$ is a fibration over a simply connected base. Let $a\in H^2(F)$
be a cohomology class of finite cup-length. If $d_2(a)=~0$ then
$d_3(a)=0$ (cf. proof of Proposition 3.1 in \cite{MR2115670}).
\end{remark}

\subsection{A dimension inequality}

\begin{lemma}\label{L:inequality}
Let $M$ be a simply connected closed manifold and let 
$\C H\subseteq \OP{Homeo}(M)$ be a connected group of homeomorphisms.
Let $G\subset \C H$ be a connected subgroup with finite $\pi_1(G)$.
Then
$$
\dim \B A_G^{2k}(M) \leq \dim \B A_{\C H}^{2k}(M) \leq \dim H^{2k}(B\C H).
$$
In particular, if $\B A_G^*$ is full then the homomorphism 
$H^*(B\C H)\to H^*(BG)$ induced by the action is surjective.
\end{lemma}

\begin{proof}
Consider the following commutative diagram of fibrations.
$$
\xymatrix
{
M \ar[d]_{j}\ar[r]^= & M\ar[d]^{i} \\
M_G \ar[d]_p\ar[r]^{F} & M_{\C H}\ar[d]^{\pi} \\
BG \ar[r]^{f} & B\C H \\
}
$$
Since $\pi_1(G)$ is finite, $H^2(BG)=0$ and the inclusion
of the fibre $j\colon M\to M_G$ induces an isomorphism
$j^*\colon H^2(M_G) \to H^2(M).$ It follows from Lemma
\ref{L:surjectivity}
that the homomorphism $F^*\colon H^2(M_{\C H})\to H^2(M_G)$ is
surjective and hence we have 
$$
p_!(a_1\ldots a_{n+k}) = p_!(F^*(\tilde a_1\ldots \tilde a_{n+k}))
= f^*(\pi_!(\tilde a_1\ldots \tilde a^{n+k}))
$$
which finishes the proof.
\end{proof}

\subsection{Cohomologically symplectic manifolds and coupling classes}
\label{SS:c-symplectic}
A closed $2n$-manifold $M$ is called {\bf{\em cohomologically symplectic}}
or shortly {\bf{\em c-symplectic}} if there exists a class 
$\alpha \in H^2(M)$ such that $\alpha ^n \neq 0.$
Such a class $\alpha$ is called a {\bf{\em symplectic class}}.

Assume that $M$ is simply connected.  Let a topological group $G$ act
on $M$ preserving a symplectic class $\alpha$.  Let 
$$
M\stackrel{i}\to E\stackrel{\pi}\to B
$$ 
be a fibration with the
structure group $G.$ There exists a unique cohomology class
$\Omega_{E}\in H^2(E)$ such that $i^*\Omega_E = \alpha$ and
$\pi_!(\Omega^{n+1}_{E})=~0.$ The class $\Omega_{E}$ is
called the {\bf{\em coupling class}}.
It is natural in the sense that the coupling class of a pull
back bundle is the pull back of the coupling class.
The symplectic class $\alpha\in H^2(M)$ is said to satisfy the 
{\bf{\em Hard Lefschetz condition}} if the multiplication by its $k$-th
power defines an isomorphism
$
H^{n-k}(M) \to H^{n+k}(M)
$
for $k=0,1,\ldots,n.$

\begin{example}
All K\" {a}hler manifolds (e.g. coadjoint orbits)
satisfy the Hard Lefschetz condition \cite{MR1288523}.\qed
\end{example}

\subsection{Consequences of the Hard Lefschetz condition}
The following lemma was proved first for complex algebraic manifolds
by Blanchard \cite{MR0087184}. The proof the following topological
version of the lemma can be found in Lalonde--McDuff \cite{MR1941438}.

\begin{lemma}[Blanchard \cite{MR0087184}]
Let $M$ be a closed simply connected c-symplectic $2n$-manifold
satisfying the Hard Lefschetz condition.
If $M\to E \to B$ is a bundle with a connected
structure group $\C H\subset \OP{Homeo}(M)$ then the homomorphism
$i^*\colon H^*(E) \to H^*(M)$ induced by the inclusion of the
fibre is surjective.\qed
\end{lemma}

Notice that the surjectivity of the homomorphism $i^*$ in
the above lemma implies, due to the Leray-Hirsch theorem
\cite[Theorem 4D.1]{MR1867354}, that $H^*(E)$ is isomorphic as a $H^*(B)$-module
to the tensor product $H^*(B)\otimes H^*(M).$ In particular,
the homomorphism $p^*\colon H^*(B)\to H^*(E)$ induced by the
projection is injective.

The next proposition is motivated by the fact that if a cohomology
class of a space $X$ evaluates nontrivially on a sphere then it is
indecomposable. That is, it cannot be expressed as the sum of products
of classes of positive degree.  Hence one can think of such a class as
a generator of the cohomology ring of $X$.

\begin{proposition}\label{P:spheres}
Let $M$ be a closed simply connected c-symplectic $2n$-manifold
satisfying the Hard Lefschetz conditions.  Let 
$$
M\stackrel{j}\to E \stackrel{p}\to S^{2k}
$$ 
be a bundle over a sphere of positive dimension and
with a connected structure group 
$\C H\subset \OP{Homeo}(M).$ Let $\sigma \in H^{2k}(S^{2k})$ denote a
generator. If $p^*(\sigma) = \sum a\cdot b,$ where all $a,b\in H^*(E)$ are
of positive degree, then the homomorphism 
$$
f^*\colon H^{2k}(B\C H) \to H^{2k}(S^{2k})
$$ 
induced by the classifying map is surjective.
\end{proposition}

\begin{proof}
Since $M$ satisfies the Hard Lefschetz condition we have
$p^*(\sigma)\neq~0.$

Next observe that the base sphere has to be of dimension bigger than
two. Indeed, if $k=1$ then $p^*(\sigma) = \sum a\cdot b$ for 
$a,b\in H^1(E).$ Since $E$ is simply connected it implies that
$p^*(\sigma)=0,$ which cannot happen.  Consequently we have 
$k>1.$ 

Notice that the homomorphism $j^*\colon H^m(E) \to H^m(M)$
induced by the inclusion of the fibre is an isomorphism for $m<2k.$

Let $M\stackrel{i}\to M_{\C H}\stackrel{\pi}\to B\C H$ be the
universal fibration and let $\Omega\in H^2(M_{\C H})$ be the coupling
class associated with the symplectic class ~$\alpha$.  Let 
$\hat a,\hat b\in H^*(M_{\C H})$ be such that $i^*\hat a = j^*a$ and
$i^*\hat b = j^*b.$

In the following calculation, $\Omega_E = F^*(\Omega)$ denotes
the coupling class. Also, since $j^*$ is an isomorphism
in degrees smaller than $2k,$ we have $F^*(\hat a) = a$ and
$F^*(\hat b) = b.$ This implies the second equality.

\begin{eqnarray*}
f^*\pi_!\left (\Omega^n \cdot \sum \hat a \cdot \hat b\right ) &=&
p_!\left (F^*\left (\Omega^n \cdot \sum \hat a \cdot \hat b\right )\right ) \\
&=& p_!\left (\Omega_E^n \cdot \sum a\cdot b\right )\\
&=& p_!(\Omega_E^n \cdot p^*\sigma) = \sigma \cdot \OP{volume}(M).
\end{eqnarray*}
\end{proof}

%%%%%%%%%%%%%%%%%%%%%%%%%%%%%%%%%%%%%%%%%%%%%%%%%%%%%%%%%%%%%%
\section{Basic properties of the fibre integral subalgebra}
\label{S:basic}
%%%%%%%%%%%%%%%%%%%%%%%%%%%%%%%%%%%%%%%%%%%%%%%%%%%%%%%%%%%%%

Throughout this section $M$ and $N$ are assumed to be  closed,
connected and simply connected manifolds.

\begin{proposition}\label{P:basic_HG}
Let $H\to G\to \OP{Homeo}_0(M)$ be a sequence of actions of connected
topological groups on a manifold $M.$ Let $f\colon BH\to BG$ denote
the induced map.  If $f^*\colon H^2(BG)\to H^2(BH)$ is surjective then
$$
\B A^*_H(M) \subset f^*(\B A_G^*(M)).
$$
\end{proposition}

\begin{proof}
Consider the diagram of universal fibrations.
$$
\xymatrix
{
M \ar[d]_{j}\ar[r]^= & M\ar[d]^{i} \\
M_H \ar[d]_p\ar[r]^{F} & M_{G}\ar[d]^{\pi} \\
BH \ar[r]^{f} & BG \\
}
$$
Let $p_!(\Omega_1 \ldots \Omega_k) \in \B A^*_H(M)$,
where $j^*(\Omega_i)=a_i$. According to Lemma~\ref{L:surjectivity},
there are classes $\widehat \Omega_i\in H^2(M_G)$ such that
$i^*(\widehat \Omega_i)=a_i.$ Since $H$ and $G$ are connected their
classifying spaces are simply connected and we have
$$
\Omega_i - F^*(\widehat \Omega_i) = p^*(\alpha_i)
$$
for some $\alpha_i \in H^2(BH)$. It follows from the hypothesis
that $\alpha_i = f^*(\beta_i)$ and we get 
$$
\Omega_i = F^*(\widehat \Omega_i - \pi^*(\beta_i)).  
$$
We finally have 
$$
p_!(\Omega_1\ldots\Omega_k) = 
f^*\pi_!((\widehat \Omega_1-\pi^*\beta_1)\ldots 
(\widehat \Omega_k-\pi^*\beta_k))
$$
which finishes the proof.

\end{proof}

\begin{proposition}\label{P:basic_product}
Let $G$ and $H$ be connected groups acting on manifolds $M$
and $N$ respectively. Then $G\times H$ acts on $M\times N$ and
$$
\B A^*_{G\times H}(M\times N)\cong \B A^*_G(M) \otimes \B A^*_H(N).
$$
In particular, if both $\B A_G^*(M)$ and $\B A^*_H(N)$ are full
then $\B A_{G\times H}^*(M\times N)$ is also full.
\end{proposition}

\begin{proof}
The statement is true due to the multiplicativity property
of the fibre integration with respect to the cross product
and the isomorphism $H^*(BH\times BG) = H^*(BH)\otimes H^*(BG).$
\end{proof}

\begin{proposition}\label{P:basic_cup}
Let a connected group $G$ act on $M$ and $N$. The cup product in
$H^*(BG)$ induces a map
$$
\B A^*_G(M)\otimes \B A^*_G(N) \to H^*(BG)
$$
with the image equal to $\B A_G^*(M\times N).$
In particular, if $G$ acts on $N$ trivially then 
$
\B A^*_G(M) = \B A^*_G(M\times N).
$
\end{proposition}

\begin{proof}
The map in the statement is the composition
$$
\B A^*_G(M)\otimes \B A^*_G(N) \stackrel{\cong} \longrightarrow 
\B A^*_{G\times G}(M\times N)\stackrel{\Delta^*}
\longrightarrow H^*(BG)
%\B A_G^*(M\times N).
$$
where the first isomorphism is due to Proposition
\ref{P:basic_product} and the second map is induced by the diagonal
$\Delta \colon BG \to BG\times BG.$

It follows from Proposition \ref{P:basic_HG} that
$\B A_G^*(M\times N)$ is contained in the image of the
above map. Thus we need to show that the converse
inclusion holds,
$$
\Delta^*(\B A^*_{G\times G}(M\times N)) 
\subset \B A^*_G(M\times N).
$$
According to simple connectivity we have that
$$
H^2(M_G\times N_G)= H^2(M_G)\oplus H^2(N_G)
$$ 
and hence an element in the subalgebra of $H^*(M_G\times N_G)$ 
generated by degree two classes is a
sum of products of the form
$
(\alpha_1\ldots \alpha_k)\times (\beta_1\ldots \beta_l),
$
for $\alpha_i\in H^2(M_G)$ and $\beta_i\in H^2(N_G).$
Since this is itself a product of degree two 
classes we get that its pull back via the map
$\widehat \Delta \colon (M\times N)_G\to M_G\times N_G$
is a product of degree two classes. This, according to
the functoriality of the fibre integration, finishes
the proof.
\end{proof}

\begin{lemma}\label{L:basic_powers}
Suppose that $G$ is a connected group with finite $\pi_1(G)$ 
acting on a closed simply connected $2n$-manifold $M.$ Then the fibre
integral subalgebra $\B A^*_G(M)$ is generated by the fibre integrals
of powers, i.e. by the classes of the form $\pi_!(a^m).$
\end{lemma}

\begin{proof}
Let $A\in H_{2k}(BG)$ be a homology class.  A certain nonzero multiple
of $A$ is represented by a map $f\colon B\to BG$ defined on a closed
oriented connected $2k$-manifold $B.$ Suppose that
$$
f^*(\pi_!(a_1\ldots a_{n+k}))\neq 0,
$$ 
where $a_i\in H^2(M_G).$ 
The map $f$ induces a bundle $M\to E\to~B$ and the above inequality is
equivalent to 
$$
0\neq F^*(a_1\ldots a_{n+k})\in H^{2(n+k)}(E)=~\B R,
$$ where 
$F\colon E\to M_G$ is the induced map of total spaces. 
It follows that the product map 
$$
F^*(H^2(M_{\C H}))\otimes \ldots \otimes F^*(H^2(M_{\C H}))\to H^{2(n+k)}(E)=~\B R
$$
is nontrivial. Since it is a polynomial map, due to the polarisation
formula, the power map $F^*(a)\mapsto F^*(a)^{n+k}$ is also
nontrivial.

This shows that 
$
f^*(\pi_!(a^{n+k}))\neq 0
$
for some class $a\in H^2(M_G).$

We have shown that for every homology class in $A\in H_{2k}(BG)$
which evaluates nontrivially on the fibre integral subalgebra
there exists a  class $a\in H^2(M_G)$ such that the fibre
integral of its $(n+k)$-th power evaluates nontrivially on
$A.$ This proves that such fibre integrals generate
$\B A^*_G(M).$
\end{proof}

\begin{remark}\label{R:powers}
Since the power map $F^*(H^2(M_{\C H}))\to H^{n+k}(E)$ in the above
proof is polynomial and nontrivial, there exists a nonempty open and
dense subset $U\subset F^*(H^2(M_{\C H}))$ such that $a^{n+k}\neq 0$
for $a\in U.$
\end{remark}

%%%%%%%%%%%%%%%%%%%%%%%%%%%%%%%%%%%%%%%%%%%%%%%%%%%%%%
\section{Calculations for coadjoint orbits}
\label{S:subalgebra}
%%%%%%%%%%%%%%%%%%%%%%%%%%%%%%%%%%%%%%%%%%%%%%%%%%%%

\subsection{Symplectic preliminaries}\label{SS:symplectic}
Let $G$ be a compact connected semi\-simple Lie group and let 
$\xi \in \mathfrak g^{\vee}$ be a covector. The coadjoint orbit
$G\cdot~\xi$ admits a $G$-invariant symplectic form.  The Killing form
provides an equivariant isomorphism between the Lie algebra 
$\mathfrak g$ and its dual $\mathfrak g^{\vee}$ and hence also a
bijective correspondence between adjoint and coadjoint orbits.

Let $T\subset G$ be a maximal torus and denote by $\mathfrak t$ 
its Lie algebra. Every adjoint orbit has a representative in
$\mathfrak t.$ The Lie algebra $\mathfrak t$ is decomposed into the
Weyl chambers. Let $C\subset \mathfrak t$ denote the closure of a Weyl
chamber. It is a polyhedral cone.  If two elements 
$\xi,\eta\in \mathfrak t$ belong to the interior of a face of $C$ then
the corresponding adjoint orbits $G\cdot \xi$ and $G\cdot \eta$ are
diffeomorphic. In this case the isotropy groups $G_{\xi}$ and
$G_{\eta}$ are conjugate in $G$.  The conjugation by an element of $G$
provides a $G$-equivariant diffeomorphism between the orbits 
$G\cdot \xi$ and $G\cdot \eta$.

Thus we can fix one orbit $M$ and consider it as a smooth manifold
equipped with various $G$-invariant symplectic forms.  Since the
conjugation induces a map of $BG$ homotopic to the identity the
universal fibration $M\to M_G\to BG$ is Hamiltonian with respect to
these symplectic forms. In such a case we have the coupling class
$\Omega_{\xi}\in H^2(M_G)$ corresponding to the symplectic form on the
orbit $G\cdot \xi \cong M$.

\subsection{Flag manifolds}\label{SS:flag}
Let $\dim G/T=2n$ and consider the universal bundle
$
G/T \to BT \stackrel{\pi}\to BG
$ 
associated with the action. 

\begin{lemma}\label{L:flag}
The fibre integral subalgebra $\B A^*_G(G/T)$ is full.
Moreover, it is generated by the fibre integrals of
powers of coupling classes.
\end{lemma}

\begin{proof}
Let $\dim G/T=2n$ and let $G/T \to BT \stackrel{\pi}\to BG$ be the
universal bundle associated with the action. Consider the following
composition
$$
H^2(BT)\otimes \dots \otimes H^2(BT) 
\to H^{2(n+k)}(BT) \stackrel{\pi_!}\to H^{2k}(BG),
$$ 
where the first map is the product and the second is the fibre
integration. Observe the first map is polynomial and surjective since
$H^*(BT)$ is the polynomial algebra generated by $H^2(BT).$ The second
map is surjective which follows from the injectivity of 
$\pi^*\colon H^*(BG) \to H^*(BT).$ Indeed, if $b \in H^{2k}(BG)$ then
$ \pi_!(\pi^*(b)\cdot \Omega^n)=b, $ for a coupling class 
$\Omega \in H^2(BT).$ This proves that the fibre integral subalgebra
is full.

The second statement follows from Lemma \ref{L:basic_powers} and
Remark \ref{R:powers}.  Indeed, there exists an open and dense subset
of $H^2(BT)$ consisting of coupling classes.
\end{proof}

\subsection{Fibre integral of a power of the coupling class as
an invariant polynomial}

The cohomology of the classifying space of a compact Lie group is
isomorphic to the algebra of invariant polynomials on the Lie algebra
$$
H^{2k}(BG)\cong S^k(\mathfrak g^{\vee})^G.
$$ 
The latter is isomorphic to $S^k(\mathfrak t)^{W_G},$ the polynomials
on the Lie algebra of the maximal torus invariant under the Weyl group
of $G.$
The next lemma follows from \cite[Lemma 3.6 and Lemma 3.9]{MR2115670}.

\begin{lemma}
Let $M=G\cdot \xi$ be a $2n$-dimensional coadjoint orbit of a
semisimple Lie group $G.$ The fibre integral of the $(n+k)$-th power
of the coupling class $\Omega_{\xi}\in H^2(M_G)$ 
corresponds to the following invariant polynomial.
$$
P_{k}(\xi,X):= (-1)^k{{n+k} \choose k}\cdot 
\int_{G}\left<X,\operatorname{Ad}^{\vee}_g(\xi)\right>^k 
\operatorname{vol}_G
$$
\qed
\end{lemma}

Since the polynomials $P_k(\xi,X)$ depend continuously (with respect
to the Zariski topology) on $\xi$ and since the algebraic independence is
an open condition we obtain the following result.

\begin{proposition}\label{P:close_orbit}
Let $\xi\in \mathfrak g^{\vee}.$
There exists a Zariski open neighbourhood $Z\subset \mathfrak g^{\vee}$
of $\xi$ such that 
$$\B A^*_G(G\cdot \xi)\subset \B A^*_G(G\cdot \eta)$$ 
for every $\eta \in Z.$ 
\qed
\end{proposition}

\begin{corollary}\label{C:generic}
If $M=G\cdot \xi\subset \mathfrak g^{\vee}$ is a generic coadjoint orbit of 
a compact semisimple Lie group $G$ then
the fibre integral subalgebra 
$\B A^*_G(M)$ is full.
\end{corollary}
\begin{proof}
It follows from Proposition \ref{P:close_orbit} that if there is
an orbit with the full fibre integral subalgebra then a generic
coadjoint orbit has full fibre integral subalgebra. The statement
follows from Lemma~\ref{L:flag}.
\end{proof}

%%%%%%%%%%%%%%%%%%%%%%%%%%%%%%%%%%%%%%%%%%%%%%%%%%%%%%%%%%%%%%%%%%%%%%
\section{Proofs of Theorem \ref{T:orbits} and Theorem \ref{T:circle}}
\label{S:proofs}
%%%%%%%%%%%%%%%%%%%%%%%%%%%%%%%%%%%%%%%%%%%%%%%%%%%%%%%%%%%%%%%%%%%%%%

\subsection{Proof of Theorem \ref{T:orbits}}\label{SS:proof_orbits}
According to Corollary \ref{C:generic} we have that the fibre integral
subalgebra $\B A^*_G(M)$ is full for a generic coadjoint orbit $M$.
Then, since $G$ is semisimple its fundamental group is finite, the
first statement follows from Lemma~\ref{L:inequality}.  

To prove the second statement observe first that a compact semisimple
group $G$ is finitely covered by a product $G_1\times \ldots \times G_m$
of simple groups. Thus there is a splitting of the Lie algebra
$\mathfrak g = \mathfrak g_1\oplus \ldots \oplus \mathfrak g_m.$

The composition 
$G_1\times \ldots \times G_m\to G \to \OP{Aut}(\mathfrak g^{\vee})$
of the covering projection and the coadjoint action is the coadjoint
action of the product and hence it is the product of coadjoint
actions $\OP{Ad}^{\vee}\colon G_i \to \OP{Aut}(\mathfrak g_i^{\vee}).$
Thus $M$ is diffeomorphic to a product $M_1\times \ldots \times M_m$
of the corresponding coadjoint orbits of simple groups.

Next we want to show that each of the above orbits $M_i$ is
of positive dimension. This follows from the assumption on
the kernel of the action. Indeed,
since the kernel of the action 
$\OP{Ad}^{\vee}\colon G\to \OP{Aut}(\mathfrak g^{\vee})$
is finite, so is the kernel of the coadjoint action
of the product. But the latter is isomorphic to the product
of the kernels of $\OP{Ad}^{\vee}\colon G_i\to \OP{Aut}(\mathfrak g^{\vee})$
and hence each $M_i$ is a nontrivial orbit.

The statement now follows from Theorem \ref{T:circle} proof of which
we shall present next.  \qed

\subsection{Proof of Theorem \ref{T:circle}}\label{SS:proof_circle}
Let $\dim M = 2n,$ $\dim M_i = 2n_i$ and let 
$\omega_a = \sum a_i\omega_i,$ where $a:=(a_1,\ldots a_m)$ is an
$m$-tuple of positive real numbers. Let $\Omega_a\in H^2(M_{\C H})$ be
the coupling class associated with the symplectic form $\omega_a.$

$$
\xymatrix
{
M_1\times \ldots \times M_m \ar[d] \ar[r]^= & M\ar[d]\\
(M_1)_{\Ham} \times \ldots \times (M_m)_{\Ham} \ar[d]^p 
\ar[r]^{\phantom{uuuuuuuuuuuu}F} & M_{\C H}\ar[d]^{\pi}\\
B{\Ham(M_1,\omega_1)} \times \ldots \times B{\Ham(M_m,\omega_m)} 
\ar[r]^{\phantom{uuuuuuuuuuuuuuuuuu}f} & B{\C H}\\
}
$$
The pull-back of the coupling class $F^*(\Om_a)$ is equal to the
sum of coupling classes $\sum a_i \Omega_i.$ 
In the following computation $C_k$ is a positive constant 
depending on a tuple $k=(k_1,\ldots,k_m)$ of non-negative integers
and the fourth equality follows from the multiplicativity of the fibre
integration with respect to the cross product.
\begin{eqnarray*}
f^*\left(\pi_!\left(\Omega_a\right)^{n+2}\right)
&=& p_!\left (F^*(\Omega_a)^{n+2}\right)\\
&=& p_!\left(\left(\sum a_i\Omega_i\right )^{n+2}\right )\\
&=& \sum C_k\cdot p_!\left ((a_1\Om_1)^{k_1}\times \ldots \times (a_m\Om_m)^{k_m}\right)\\
&=& \sum C_k\cdot (p_1)_!((a_1\Om_1)^{k_1})\times \ldots \times (p_m)_!(a_m\Om_m)^{k_m})\\
\end{eqnarray*} 
Notice that the fibre integral $(p_i)_!(\Om_i^k)=0$ if $k<n_i.$ Thus
the above sum is,
up to a positive constant, equal to the sum of terms of the form
$$
(p_1)_!(a_1\Om_1)^{n_1}\times \ldots \times (p_i)_!(a_i\Om_i)^{n_i+2}\times 
\ldots \times (p_m)_!(a_m\Om_m)^{n_m}.
$$

It follows from \cite[Theorem 1.1]{MR2115670} that an inessential
nontrivial circle action induces an element of 
$\sigma_i\in \pi_4(B\Ham(M_i))$ on which the class
$(p_i)_!(\Omega_i^{n+2})$ evaluates nontrivially.  By varying the
$m$-tuple $a=(a_1,\ldots,a_m)$ we obtain that the image of the
homomorphism
$$
f^*\colon H^4(B\C H)\to 
H^4(B\Ham(M_1,\omega_1)\times \ldots \times B\Ham(M_m,\omega_m))
$$ 
is at least $m$-dimensional. 

Choosing the parameters $a$ appropriately,
the classes $\pi_!(\Omega_a^{n+2})$ define $m$ linearly independent
functionals on $\pi_4(B\C H)\otimes \B R.$ Evaluating them on the
images of the classes $\sigma_i$ we obtain that the
rank $\pi_4(B\C H)$ is at least $m.$ This rank is equal
to the rank of $\pi_3(\C H).$ 
\qed

%%%%%%%%%%%%%%%%%%%%%%%%%%%%%%%%%%%%%%%%%%%%%%%%%%%%%%
\section{Examples}\label{S:examples}
%%%%%%%%%%%%%%%%%%%%%%%%%%%%%%%%%%%%%%%%%%%%%%%%%%%%%%

Let $G$ be a compact connected semisimple Lie group with a maximal
torus $T.$ Let $\mathfrak g$ and $\mathfrak t$ denote the
corresponding Lie algebras.  The closed positive Weyl chamber
$C\subset \mathfrak t$ is a simplicial cone.  Let $F$ be a face of
$C.$ If $\xi$ and $\eta$ belong to the interior of $F$ then they are
diffeomorphic. Moreover, if $\xi \in \OP{interior}(F)$ then
$$
\dim H^2(G\cdot \xi) = \dim F.
$$ 
Conversely, a coadjoint orbit $G\cdot \xi$ 
has a representative that belongs to the interior of a face
of dimension equal to $\dim H^2(G\cdot \xi).$

\begin{remark}
The above observations can be easily deduced from Bott's results in
\cite{MR0087035} (see also Section 2 of his survey in \cite{MR568880}
and Section 2.3 of Guillemin--Lerman--Sternberg \cite{MR98d:58074}).
\end{remark}

Proposition \ref{P:close_orbit} states that $\B A^*_G(G\cdot \xi)
\subset \B A^*_G(G\cdot \eta)$ for any $\eta$ in an open and dense
neighbourhood $U\subset C$ of $\xi.$ Thus the most interesting are the
orbits corresponding to the edges (i.e. one-dimensional faces) of $C.$
This is because near an edge there are points corresponding to
many topologically distinct orbits.  The orbits corresponding to edges
are characterised by their second Betti number being equal to one.
They are called {\bf \em minimal} in the terminology of
Guillemin--Lerman--Sternberg \cite{MR98d:58074}.  

\begin{example}
The minimal coadjoint orbits of $\OP{SU}(n)$ 
are the complex grassmannians $\OP{G}(k,n).$
\end{example}

\subsection{Special unitary group $\OP{SU}(n)$}\label{SS:sun}
It is known that the fibre integral subalgebra associated
with the natural projective action of $\OP{SU}(n)$ on the
complex projective space $\cp^{n-1}$ is full 
\cite[Proposition 1.7]{MR2115670}.
Thus it follows from Proposition \ref{P:close_orbit}
that every coadjoint orbit $M$ of $\OP{SU}(n)$ close to
the projective space has full fibre integral subalgebra
$\B A^*_{\OP{SU}(n)}(M).$ Such orbits are of the form
$$
\OP{SU}(n)/
\OP{S}\,(\OP{U}(n_1)\oplus \ldots \oplus \OP{U}(n_k)\oplus \OP{U}(1))
$$
where $n_1+\ldots + n_k+1 = n$.

\begin{example}[Proposition 3.6 in \cite{gkt}]\label{E:g24}
The fibre integral subalgebra of the complex grassmannian 
$\OP{G}(n,2n)$ with respect to the natural action
of $\OP{SU}(2n)$ is not full. More precisely, 
$H^6(B\OP{SU}(2n))$ is not contained in the fibre integral
subalgebra. This is due to the fact that the relevant invariant
polynomial has odd degree and its zero represents the
grassmannian (see \cite[Section 3]{gkt} 
for more examples and details).
\end{example}

\begin{remark}
It is shown in \cite[Proposition 4.8]{MR2115670} that
the natural action of $\OP{SU}(n)$ on a generalised flag manifold $M$
induces a surjective homomorphism
$H^*(B\Homeo_0(M)) \to H^*(B\OP{SU}(n)).$ The proof is also an
application of fibre integrals. It is, however, specialised to this
particular case.
\end{remark}

\subsection{Special orthogonal group $\OP{SO}(2m), m>2$}
The cohomology ring of the classifying space of the special orthogonal
group is generated by the Pontryagin classes $p_1,\ldots,p_m$ and the
Euler class $e.$ They have degrees $\deg(p_k)=4k$ and $\deg(e)=2m.$

Consider a minimal coadjoint orbit of the form $M:=\OP{SO}(2m)/\OP{U}(m).$ The
classes $p_1,e,p_{m+1},\ldots,p_{2m}$ belong to the fibre integral
subalgebra $\B A_{\OP{SO}(2m)}^*(M).$ For the first Pontryagin class
this follows from the second part of Theorem \ref{T:orbits}.

To see that the the higher Pontryagin classes $p_k$ for $k>m$ belong
to the fibre integral subalgebra consider the map 
$\pi\colon B\OP{U}(m)\to B \OP{SO}(2m)$ induced by the inclusion. The
pullback of a Pontryagin class is expressed in terms of Chern classes
by the well-known formula
$$
\pi^*(p_k) = c_k^2 - 2c_{k-1}c_{k+1} + \ldots \pm 2c_{2k}.
$$
Thus if $k>m$ then $\pi^*(p_k)$ is a sum of products of classes of
positive degrees and Proposition \ref{P:spheres} applies.

Finally, a result of Reznikov~\cite{KTW,MR2000f:53116} implies that
the Euler class belongs to the fibre integral subalgebra.

Since $H^2(M)=\B R$ there are, in general, orbits close to $M$
that are topologically different from $M.$ Their fibre integral
subalgebras contain $\B A^*_{\OP{SO}(2m)}(M).$

\begin{remark}
We excluded the case of $\OP{SO}(4)$ because then the action
of the group $\OP{SO}(4)$ on the orbit $\OP{SO}(4)/\OP{U}(2)=\cp^1$
is not effective.
\end{remark}

\section{The fibre integral subalgebra is stable}
If $G$ acts on a manifold $M$ then it also acts on
the product $M\times N$ (acting trivially on the second factor).
We think of the composition of actions
$$
f\colon G \to \Homeo(M)\to \Homeo(M\times N)
$$ 

as a kind of stabilisation. 
If $G$ is connected then the second part
of Proposition \ref{P:basic_cup} states that 
$\B A^*_G(M) = \B A_G^*(M\times N)$ and hence
$$
\B A^*_G(M) \subset f^*(B\OP{Homeo}_0(M\times N)).
$$
This means that the part of the topology of the classifying
space of the group of homeomorphisms of $M$ captured by the
fibre integral subalgebra persists when we stabilise $M.$

\begin{example}
Let $M$ be a coadjoint orbit of a semisimple compact Lie group $G$.
Let $N$ be a closed simply connected symplectic manifold.
Then the fibre integral subalgebra $\B A^*_G(M)$ is contained
in the image of the homomorphism
$$
H^*(B\Ham(M\times N))\to H^*(BG)
$$
induced by the action of $G$ on the product.
In particular if the fibre integral subalgebra is full then
the above homomorphism is surjective.
\end{example}

\begin{remark}
The fundamental group of the group of Hamiltonian diffeomorphisms
of a product symplectic manifold has been recently investigated
by Pedroza in \cite{MR2440331}.
\end{remark}

\bibliography{../bib/bibliography}
\bibliographystyle{plain}

\end{document}